\RequirePackage{amsmath}
\documentclass[runningheads]{llncs}
\usepackage{amssymb}
\usepackage{mathtools}
\usepackage[hidelinks]{hyperref} %remove hidelinks at the end
\usepackage{xcolor}
\usepackage[capitalize,nameinlink,sort]{cleveref}
\usepackage{tikz}
\usetikzlibrary{automata}
\usepackage{subcaption}
\usepackage{graphicx}
\usepackage{ragged2e}

\newcommand{\N}{\mathbb{N}}
\newcommand{\eps}{\varepsilon}
\DeclareMathOperator{\rep}{rep}
\DeclareMathOperator{\val}{val}
\newcommand{\infw}[1]{\mathbf{#1}}

\makeatletter
\newcommand*{\rom}[1]{\expandafter\@slowromancap\romannumeral #1@}
\makeatother

\usepackage{float}

\begin{document}
\title{On the pseudorandomness of Parry--Bertrand automatic sequences
%\thanksw{}
}
\titlerunning{On the pseudorandomness of Parry--Bertrand automatic sequences}
\author{Pierre Popoli\orcidID{0000-0002-4243-9180}
\and
Manon Stipulanti\orcidID{0000-0002-2805-2465}}
\authorrunning{P. Popoli and M. Stipulanti}
\institute{Department of Mathematics, ULiège,
Liège, Belgium \\
\email{\{pierre.popoli,m.stipulanti\}@uliege.be}}
\maketitle              
\begin{abstract}
%The abstract should briefly summarize the contents of the paper in 15--250 words.

The correlation measure is a testimony of the pseudorandomness of a sequence $\infw{s}$ and provides information about the independence of some parts of $\infw{s}$ and their shifts.
Combined with the well-distribution measure, a sequence possesses good pseudorandomness properties if both measures are relatively small.
In combinatorics on words, the famous $b$-automatic sequences are quite far from being pseudorandom, as they have small factor complexity on the one hand and large well-distribution and correlation measures on the other.
This paper investigates the pseudorandomness of a specific family of morphic sequences, including classical $b$-automatic sequences.
In particular, we show that such sequences have large even-order correlation measures; hence, they are not pseudorandom.
We also show that even- and odd-order correlation measures behave differently when considering some simple morphic sequences.
\end{abstract}

\section{Introduction}

Pseudorandom sequences are generated by algorithms and behave similarly to truly random sequences.
To study the pseudorandomness of a sequence, a wide variety of measures has been introduced, called \emph{measures of complexity} or \emph{measures of pseudorandomness}.
For instance, from a purely theoretical point of view, the famous \emph{Kolmogorov complexity} of a sequence gives the length of the shortest program that generates it~\cite{Kolmogorov68}; see also~\cite{Chaitin1966,Solomonoff1964}. 
An infinite sequence is said to be \emph{algorithmically random} if its Kolmogorov complexity grows like the length of the sequence itself.
%This should be the case if the latter is not too ``simple''.
It is well known that Kolmogorov complexity is not a computable function; for example, see~\cite{Vitanyi2020} for a formal proof and history behind.

As a consequence, many other measures of complexity can be used, especially for practical issues. The series of papers~\cite{MS-I,MS-II}, initiated by Mauduit and S\'{a}rk\"{o}zy in 1997, introduces and studies several such measures.
Among the most important ones, the \emph{well-distribution measure} and the \emph{correlation measure} are considered in various contexts.
Let $\infw{s}=(\infw{s}(n))_{n\geq 0}$ be a sequence over the alphabet $\{0,1\}$.
For $a\in \mathbb{Z}$ and $b,M\in \mathbb{N}$, write $U(\infw{s},M,a,b)=\sum_{j=0}^{M-1}(-1)^{\infw{s}(a+jb)}$
%\[
%U(\infw{s},M,a,b)=\sum_{j=0}^{M-1}(-1)^{\infw{s}(a+jb)}
%\]
and for $k \in \N_{\geq 1},D=(d_1,\ldots,d_k) \in \mathbb{N}^k$ such that $0\le d_1<d_2<\cdots<d_k$, write
\[
V(\infw{s},M,D)=\sum_{n=0}^{M-1}(-1)^{\infw{s}(n+d_1)+\infw{s}(n+d_2)\cdots +\infw{s}(n+d_k)}.
\]
The \emph{Nth well-distribution measure} of the sequence $\infw{s}$ is defined by
\[
W(\infw{s},N)=\max \limits_{a,b,M} \left| U(\infw{s},M,a,b) \right|=\max \limits_{a,b,M} \left|\sum_{j=0}^{M-1}(-1)^{\infw{s}(a+jb)} \right|,
\]
where the maximum is taken over all $a,b,M \in \mathbb{N}$ such that $0\leq a \leq a+(M-1)b<N$, and the \emph{$N$th correlation measure of order $k$} of $\infw{s}$ is given by
\[
C_k(\infw{s},N)
=\max \limits_{M,D} |V(\infw{s},M,D)|=\max \limits_{M,D} \left| \sum_{n=0}^{M-1}(-1)^{\infw{s}(n+d_1)+\infw{s}(n+d_2)\cdots +\infw{s}(n+d_k)} \right|, 
\]
where the maximum is taken over all vectors $D=(d_1,\ldots,d_k)\in \mathbb{N}^k$ and all integers $M$ such that $0\le d_1<d_2<\cdots<d_k$ and $M+d_k \leq N$. 

It is said that the sequence $\infw{s}$ possesses good properties of pseudorandomness if both its measures $W(\infw{s},N)$ and $C_k(\infw{s},N)$ are relatively small compared to $N$.
Indeed, Alon et al.~\cite{AKMMR2007}, improving a former result of Cassaigne, Mauduit, and S\'{a}rk\"{o}zy~\cite{CMS2001}, prove that the expected order of $W(\infw{s},N)$ and $C_k(\infw{s},N)$ when $\infw{s}$ is a truly random binary sequence is $\sqrt{N}(\log N)^{O(1)}$. Furthermore, it is crucial to simultaneously consider the two measures, since the smallness of one does not necessarily imply that of the other; see \cite{MS-I} for more details. However, it is sufficient to show that the correlation measure of a sequence is large to tell it apart from a pseudorandom sequence. Notice that the combination of the well-distribution and the correlation measures gives birth to the \emph{well-distribution-correlation measure}, which is however more complicated to study. 

%\begin{theorem}[\cite{MS-VII}]
%For any integer $k \geq 2$ and for any real $\varepsilon>0$, there exist $N_0=N_0(\varepsilon,k)$ and $\delta=\delta(\varepsilon,k)>0$ such that for any $N \geq N_0$ we have with probability $>1-2\varepsilon$, 
%\[\delta \sqrt{N} < C_k(\infw{s},N) < 5 \sqrt{kN\log N}. \]
%\end{theorem}

% \begin{theorem}[\cite{AKMMR2007}]
% For any real $\varepsilon>0$, there exists an integer $N_0=N_0(\varepsilon)$ such that for any $N\geq N_0$, we have with probability $1-\varepsilon$, \[\frac{2}{5}\sqrt{N\log \binom{N}{k}}<C_k(\infw{s},N)<\frac{7}{4}\sqrt{N\log \binom{N}{k}}\] for any integer $k$ such that $2 \leq k \leq N/4$.
% \end{theorem}

Usually, the methods to analyze correlation measures are borrowed from number theory; see~\cite{Gyarmati2004,GM2012} and the references therein for general results. 
However, in this paper, we study a large family of sequences through the lens of their correlation measures by mixing tools from automata theory and general numeration systems.
Such a family is that of \emph{automatic sequences}; see~\cref{sec:background} for precise definitions.
In combinatorics on words, they are classical and have many beautiful properties.
Among others, they are known to be bad pseudorandom sequences as explained hereafter. 
First, it is well known that their \emph{factor complexity}, which counts the number of different blocks of each length appearing in the whole sequence, is bounded by a linear function and is, therefore, small.
On the other hand, over an alphabet of size $\ell$, the $n$th term of the factor complexity of a random sequence is of order $\ell^n$; see~\cite[Chapter~10]{AS03} for more details.
Second, Mérai and Winterhof~\cite{MW18} prove that the order-$2$ correlation of an automatic sequence is large, as stated below.

\begin{theorem}[{\cite[Theorem~7]{MW18}}]
\label{thm:AS_MW}
Let $b\ge 2$ be an integer and let $\infw{s}$ be a $b$-automatic binary sequence generated by a DFAO with $m$ states. Then $C_2(\infw{s},N) \ge \frac{N}{b(m+1)}$ for all $N\ge b(m+1)$.
\end{theorem}

The focus on correlation measures of automatic sequences dates back to the previous century.
In the late 1920s, Mahler~\cite{Mahler1927} analyzes the sequence $((-1)^{\infw{s_2}(n)})_{n\geq 0}$, where $\infw{s_2} \colon \N \to \N$ is the \emph{sum-of-digit function} in base $2$. In particular, he considers, for any non-negative integer $k$, the quantity 
\begin{align*}
    c_M(k)=\frac{1}{M} V(\infw{s_2},M,D)=\frac{1}{M}\sum_{n=0}^{M-1}(-1)^{\infw{s_2}(n)+\infw{s_2}(n+k)}
\end{align*}
where $D=(0,k)$, and shows that the sequence $(c_M(k))_{M\ge 0}$ converges to some real number $c(k)$.
It turns out that this quantity is related to the order-$2$ correlation measure of the $2$-automatic Thue--Morse sequence $\infw{t}=(\infw{t}(n))_{n\geq 0}$ as it is defined by $\infw{t}(n)=\infw{s_2}(n) \bmod 2$ for all $n\ge 0$.
We also refer to~\cite{Mauduit2001} for more historical background on $\infw{t}$ and specific values of the quantity $c(k)$. 
Later on, Mauduit and S\'{a}rk\"{o}zy~\cite{MS-II} prove that $C_{2}(\infw{t},N)\geq \frac{1}{12}N$ for $N\geq 5$.
However, when~\cref{thm:AS_MW} is applied to $\infw{t}$, the lower bound is improved as $C_{2}(\infw{t},N)\geq \frac{1}{6}N$ for $N\geq 6$; see~\cite{MW18} for more details.  
In a similar context, Baake and Coons \cite{BC2023} recently study the correlations of $\infw{t}$ with specific values of the shifted vector $D$, as well as their means with renormalization techniques.
Mazáč \cite{Mazac2023} carries out an analogous approach for the correlations of the Rudin--Shapiro sequence, another distinguished automatic sequence.
We also mention the work of Grant, Shallit, and Stoll~\cite{GrantShallitStoll-2009} (and papers following this trend), where correlations over alphabets of size more than $2$ are considered.

It is known that automatic sequences are morphic; see~\cite[Chapter~6]{AS03} and~\cref{sec:background} for precise definitions.
In this more general framework, the factor complexity is bounded by a quadratic function, which, in turn, again shows that they are far from being pseudorandom; see~\cite[Chapter~10]{AS03}.
Consequently, our study is driven by the natural question to know whether the order-$2$ correlation measure of any morphic sequence is large.
As a first step towards answering this question, we presently consider a subfamily of morphic sequences based on Parry--Bertrand numeration systems (see~\cref{sec:background} for precise definitions).
In the theory of numeration systems, they have gained the attention of the scientific community for the nice framework they provide and have especially been used in~\cite{CCS2021,CCS2022,MPR2019,Point2000,Stipulanti2019}.

Our main result is~\cref{thm:main} below (see again~\cref{sec:background} for precise definitions).
It improves~\cref{thm:AS_MW} in at least two ways.
First, we consider a much larger class of sequences as classical automatic sequences are $U_\beta$-automatic for some integer $\beta$.
Second, we deal with all even-order correlations, whereas only the case of order $2$ is considered in~\cite{MW18}. Notice that we also discuss the case of odd-order correlations for which such a theorem fails to hold; see~\cref{sec:linearly rec seq}. 
Note that, for two maps $f,g$ such that $g$ takes only non-negative values, we write $f\gg g$ (resp., $f \ll g$) if there exists a positive constant $c$ such that $|f(x)| \ge c g(x)$ (resp., $|f(x)| \le c g(x)$) for all sufficiently large $x$. 

\begin{theorem}
\label{thm:main}
Let $k\geq 1$ be an integer, $\beta \in \mathbb{R}_{>1}$ be a Parry number and $U_\beta$ be the corresponding canonical Parry--Bertrand numeration.
For all binary $U_\beta$-automatic sequences $\infw{s}$, we have $C_{2k}(\infw{s},N) \gg N$, where the constant depends on $\beta$, $k$, and the size of the automaton generating $\infw{s}$. 
\end{theorem}

%%%%%
%%%%%%%%
%%%%%%%%

The present paper is organized as follows.
In~\cref{sec:background}, we recall the necessary background on combinatorics on words, automatic sequences, and numeration systems.
\cref{sec:our result} is divided into two parts.
In~\cref{sec:linearly rec seq}, we deal with the case of linearly recurrent sequences.
In particular, we highlight the fact that even- and odd-order correlation measures behave differently, which will imply that there is no hope of extending~\cref{thm:main} to odd-order correlations.
In~\cref{sec:proof of main}, we show with~\cref{thm:main} that our specific morphic sequences have a large correlation measure of even orders.
Finally,~\cref{sec:conclusion} presents possible future work.

%%%%%%%%
%%%%%%%%
%%%%%%%%
%%%%%%%%
%%%%%%%%
%%%%%%%%
%%%%%%%%
%%%%%%%%
%%%%%%%%
%%%%%%%%
%%%%%%%%
%%%%%%%%

\section{Background}\label{sec:background}

\subsection{On combinatorics on words}\label{sec:cow}

As a general reference on words, we point out~\cite{Loth97}.
An \emph{alphabet} is a finite set of elements called \emph{letters}.
A \emph{word} on an alphabet $A$ is a sequence of letters from $A$. It is either \emph{finite} or \emph{infinite}.
In order to distinguish them, infinite words are written in bold.
The \emph{length} of a finite word, denoted between vertical bars, is the number of letters it is made of.
The \emph{empty word} is the only $0$-length word.
For all $n\ge 0$, we let $A^n$ denote the 
set of all length-$n$ words over $A$.
We let $A^*$ denote the set of finite words over $A$, including the empty word, and $A^\omega$ or $A^\N$ that of infinite words over $A$.
The first is equipped with the concatenation of words.
%which makes it a monoid (i.e., called the \emph{free monoid} on $A$).
Let $A$ and $B$ be finite alphabets.
A \emph{morphism} $f \colon A^* \to B^*$ is a map satisfying $f(uv)=f(u)f(v)$ for all $u,v\in A^*$.
In particular, $f(\eps)=\eps$.
%and $f$ is entirely determined by the images of the letters in $A$.
For an integer $\ell\ge 2$, a morphism is \emph{$\ell$-uniform} if it maps each letter to a length-$\ell$ word.
A $1$-uniform morphism is called a \emph{coding}.
An infinite word $\infw{x}$ is \emph{morphic} if there exist a morphism $f\colon A^* \to A^*$, a coding $g\colon A^* \to B^*$, and a letter $a \in A$ such that $\infw{x} = g(f^{\omega}(a))$, where $f^{\omega}(a) = \lim_{n\to \infty} f^n(a)$.

% \begin{example}\label{ex:Fibonacci-word}
%     The \emph{Fibonacci word} $\infw{f}=01001010 \cdots$ is the limit of the sequence $(f^n(0))_{n\ge 0}$ where $f\colon \{0,1\}^* \to \{0,1\}^*$ is defined by $f(0)=01$ and $f(1)=0$.
%     The first few iterations are $0,01,010,01001,01001010$.
% \end{example}

\subsection{On generalized automatic sequences}\label{sec:NS}

\emph{Abstract numeration systems} were introduced at the beginning of the century by Lecomte and Rigo~\cite{Lecomte-Rigo-2001}; see also~\cite[Chapter~3]{CANT10} for a general presentation.
Such a numeration system is defined by a triple $S=(L,A,<)$ where $A$ is an alphabet ordered by the total order $<$ and $L$ is an infinite \emph{regular} language over $A$, i.e., accepted by a deterministic finite automaton. 
We say that $L$ is the \emph{numeration language} of $S$.
When we \emph{genealogically} (i.e., first by length, then using the dictionary order) order the words of $L$, we obtain a one-to-one correspondence $\rep_S$ between $\N$ and $L$.
Then, the \emph{$S$-representation} of the non-negative integer $n$ is the $(n+1)$st word of $L$, and the inverse map, called the \emph{(e)valuation map}, is denoted by $\val_S$.

\begin{example}
  Consider the abstract numeration system $S$ built on the language $a^* b^*$ over the ordered alphabet $\{a,b\}$ with $a<b$.
  The first few words in the language are $\eps,a,b,aa,ab,bb,aaa$. 
  For instance, $\rep_S(5)=bb$ and $\val_S(aaa)=6$.
\end{example}

Let $S=(L,A,<)$ be an abstract numeration system.
An infinite word $\mathbf{x}$ is {\em $S$-automatic} if there exists a deterministic finite automaton with output (DFAO) $\mathcal{A}$ such that, for all $n\ge 0$, the $n$th term $\mathbf{x}(n)$ of $\infw{x}$ is given by the output $\mathcal{A}(\rep_S(n))$ of $\mathcal{A}$.
See~\cite{AS03,RM2002,Sha88} for general references. 
%For the case of integer base numeration systems, i.e. when $\mathcal{A}$ is fed with the genealogically ordered language $L_b= \{\eps\} \cup \{1,\ldots,b-1\}\{0,\ldots,b-1\}^*$ for some integer $b\geq 2$, then $\mathbf{x}$ is said to be {\em $b$-automatic} and a classical reference on these sequences is~\cite{AS03}. For more exotic numeration systems, we refer to~\cite{RM2002,Sha88}.

\subsection{Representation of real numbers}

We now recall several definitions and results about representations of real numbers; see, for instance,~\cite{CANT10,Rigo14-2}.
%Let $ \lceil \cdot \rceil$ denote the \emph{ceiling function} defined by $ \lceil x \rceil = \inf \{ z\in \mathbb{Z} \mid z \ge x \}$.
Let $\beta \in \mathbb{R}_{>1}$ and let $A_\beta=\{0, 1, \ldots, \lceil \beta \rceil -1 \}$.
Every real number $x \in [0, 1)$ can be written as a series $x = \sum_{j=1}^{+\infty} c_j \beta^{-j}$, where $c_j \in A_\beta$ for all $j\ge 1$. 
%The infinite word $\infw{ca}=c_1 c_2 \cdots$ is called a \emph{$\beta$-representation} of $x$. 
We let $d_\beta(x)$ denote the \emph{$\beta$-expansion} of $x$ obtained greedily.
%Among all $\beta$-representations of $x$, we define the \emph{$\beta$-expansion} $d_\beta(x)$ of $x$ obtained in a greedy way, i.e., for all $j\ge 1$, we have $c_j \beta^{-j} + c_{j+1} \beta^{-j-1} + \cdots  < \beta^{-j+1}$.
%We also make use of the following usual convention: if $w=w_n\cdots w_0$ is a finite word (resp., $w=w_1 w_2 \cdots$ is an infinite word) over $A_\beta$, the notation $0.w$ has to be understood as the real number $\sum_{j=0}^{n} w_j \beta^{j-n-1}$ (resp., $\sum_{j=1}^{+\infty} w_j \beta^{-j}$); it actually corresponds to the value of the word $w$ in base $\beta$. 
In an analogous way, the \emph{$\beta$-expansion} $d_\beta(1)$ of $1$ is $d_\beta(1) = \lim_{x\to 1} d_\beta(x)$.
We define the \emph{quasi-greedy $\beta$-expansion} $d_\beta^*(1)$ of $1$ as follows. 
Write $d_\beta (1) = (t(n))_{n\ge 1}$. 
If $d_\beta (1)= t(1) \cdots t(m) 0^\omega$ with $t(m) \neq 0$, then $d_\beta^*(1) = (t(1) \cdots t(m-1) (t(m) - 1))^\omega$; otherwise $d_\beta^*(1) = d_\beta(1)$.
We define the \emph{$\beta$-shift} $S_\beta$ as the topological closure of the set $D_\beta=\{ d_\beta(x) \mid x\in[0,1) \}$.
A real number $\beta > 1$ is a \emph{Parry number} if $d_\beta(1)$ is ultimately periodic. 
%In this case,~\cref{pro:auto} below roughly gives an easy way to decide if an infinite word is the $\beta$-expansion of a real number~\cite{Par60}. 

\begin{example}\label{ex:integer-phi-et-phi-sqr}
%If $\beta \in \mathbb{R}_{>1}$ is an integer, then $d_\beta(1)=\beta0^{\omega}$ and $d^*_\beta(1)=(\beta  -1)^\omega$. In this case, the automaton of~\cref{pro:auto} consists of a single initial and final state $a_0$ with a loop of labels $0,1,\ldots,b-1$.
Any integer $\beta>1$ is a Parry number.
The golden ratio $\varphi$ is a Parry number for which $d_\varphi(1)=11$ and $d^*_\varphi(1)=(10)^\omega$ (since $1 = 1/\varphi + 1/\varphi^2$).
%The automaton $\mathcal{A}_\varphi$ of~\cref{pro:auto} is depicted in~\cref{fig:Aut-phi}. 
The square $\varphi^2$ of the golden ratio is also a Parry number with $d_{\varphi^2}(1)=21^\omega=d^*_{\varphi^2}(1)$.
\end{example}

Parry's theorem \cite{Par60} notably provides a purely combinatorial condition to check if an infinite sequence is in $S_{\beta}$ by comparing this sequence to $d_\beta^*(1)$ using the lexicographic order. 

% \begin{proposition}[{\cite[Corollaries~2.71 and~2.72]{Rigo14-2}}]\label{pro:auto}
% Let $\beta \in \mathbb{R}_{>1}$ be a Parry number. 
% Write $d_\beta (1)= t_1 \cdots t_m (t_{m+1}\cdots t_{m+k})^\omega$ where $m,k\geq 0$ are taken to be minimal.
% Then an infinite word belongs to $S_\beta$ if and only if it is the label of a path in the automaton $\mathcal{A}_\beta=(\{a_0,\ldots,a_{m+k-1}\},$ $a_0, A_\beta, \delta, \{a_0,\ldots,a_{m+k-1}\})$, where the transition function $\delta$ is defined as follows: for each $i\in \{1,\ldots,m+k\}$, $\delta(a_{i-1},t)=a_0$ for all $t \in \{0, \ldots, t_i-1\}$; for every $i\in \{1,\ldots,m+k-1\}$, $\delta(a_{i-1}, t_i) = a_i$, and $\delta(a_{m+k-1}, t_{m+k}) = a_m$.
% \end{proposition}

\begin{proposition}[{\cite[Corollaries~2.71 and~2.72]{Rigo14-2}}]\label{pro:auto}
Let $\beta \in \mathbb{R}_{>1}$ be a Parry number. 
Write $d_\beta (1)= t(1) \cdots t(m) (t(m+1)\cdots t(m+k))^\omega$ (where $m,k\geq 0$ are taken to be minimal).
Then an infinite word belongs to $S_\beta$ if and only if it is the label of a path in the automaton $\mathcal{A}_\beta=(Q,a_0, A_\beta, \delta, Q)$, where $Q=\{a_0,\ldots,a_{m+k-1}\}$ and where the transition function $\delta$ is defined as
\begin{align*}
    &\delta(a_{i-1},t)=a_0, \; \text{for all }1\leq i <m+k \text{ and all } 0\leq t<t(i); \\
    &\delta(a_{i-1},t(i))=a_i, \; \text{for all }1\leq i <m+k; \\
    &\delta(a_{m+k-1}, t(m+k)) = a_m.    
\end{align*}
\end{proposition}

Every Parry number is canonically associated with a linear positional numeration system.
We now expose this framework.
Let $U=(U(n))_{n\ge 0}$ be an increasing sequence of integers with $U(0)=1$.
Any integer $n$ can be decomposed uniquely in a greedy way as $n=\sum_{i=0}^t c_i\, U(i)$ with non-negative integer coefficients $c_i$.
The finite word $\rep_U(n)=c_t\cdots c_0$ is called the {\em (greedy) $U$-representation} of $n$. 
%The finite word $c_t\cdots c_0\in\mathbb{N}^*$ is a {\em $U$-representation} of $n$.
%If this representation is computed greedily \cite{Fraenkel85,Rigo14-2}, then for all $j< t$ we have $\sum_{i=0}^j c_i\, U(i)<U(j+1)$ and $\rep_U(n)=c_t\cdots c_0$ is said to be the {\em greedy} $U$-representation of $n$. 
By convention, the representation of $0$ is the empty word~$\eps$.
The finiteness of the \emph{digit-set} $A_U$ for greedy representations is implied by the boundedness of the sequence $\sup_{i\ge 0} U(i+1)/U(i)$. 
A sequence $U$ satisfying all the above conditions defines a \emph{positional numeration systems}.
The \emph{numeration language} is the set $L_U = \{ \rep_U(n) \mid n\ge 0\}$.
For any $c_t\cdots c_0\in\mathbb{N}^*$, we let $\val_U(c_t\cdots c_0)$ denote the integer $\sum_{i=0}^t c_i\, U(i)$.
In addition, a positional numeration system is \emph{linear} if $U$ satisfies a linear recurrence relation.
%i.e., there exist $k\ge 1$ and $a_0, \ldots, a_{k-1} \in \mathbb{Z}$ such that
%\begin{align}
%U(n+k) = a_{k-1}\, U(n+k-1) + \cdots + a_0\, U(n) \quad \forall n \ge 0.
% \label{eq:eq-rec-U}
%\end{align}

\begin{example}\label{ex:Fibonacci-Zeckendorf-ANS}
  The Fibonacci sequence $F=(F(n))_{n\ge 0}$ with initial conditions $F(0)=1$ and $F(1)=2$ and $F(n+2)=F(n+1)+F(n)$ for all $n\ge 0$ gives rise to the famous \emph{Zeckendorf} or \emph{Fibonacci} linear positional numeration system~\cite{Zeck72}.
  The numeration language is $L_F=\{\eps\}\cup 1\{0,01\}^*$.
\end{example}

Let $\beta \in \mathbb{R}_{>1}$ be a Parry number. We define a particular linear positional numeration system $U_\beta=(U_\beta(n))_{n\ge 0}$ associated with $\beta$ as follows.
Write $d_\beta (1)= t(1) \cdots t(m) (t(m+1)\cdots t(m+k))^\omega$ ($m,k$ are minimal and might be zero).
We define $U_\beta(0)=1$, $U_\beta(i)= t(1) U_\beta(i-1) + \cdots + t(i) U_\beta(0) + 1$ for all $i\in \{1,\ldots,m+k-1\}$ and
\begin{align*}
U_\beta(n)=& \; t(1) U_\beta(n-1) + \cdots + t(m+k) U_\beta(n-m-k) + U_\beta(n-k) \\ 
&- t(1) U_\beta(n-k-1) - \cdots - t(m) U_\beta(n-m-k)
\end{align*}
for all $n\ge m+k$.
We set $L_\beta = L_{U_\beta}$ for the sake of readability.

%Let $n$ be a positive integer. 
%By successive Euclidean divisions, there exists $\ell\ge 1$ such that 
%$$ 
%n = \sum_{j=0}^{\ell -1} c_j\, U(j)
%$$
%where the $c_j$'s are non-negative integers and $c_{\ell-1}$ is non-zero. 
%The word $c_{\ell -1}\cdots c_0$ is called the \emph{normal $U$-representation} of $n$ and is denoted by $\rep_U(n)$. 
%In other words, the word $c_{\ell-1}\cdots c_0$ is the greedy expansion of $n$ in the considered numeration system. 
%We set $\rep_U(0)=\varepsilon$. 
%Finally, we refer to $L_U=\rep_U(\mathbb{N})$ as the \emph{language of the numeration} and we let $A_U$ denote the minimal alphabet such that $L_U \subset A_U^*$ (since $\sup_{n \geq 0} \frac{U(n+1)}{U(n)}$ is bounded by a constant, such a finite alphabet does exist).
%If $d_{r}\cdots d_0$ is a word over an alphabet of digits, then its \emph{$U$-numerical value} is $\val_U(d_{r}\cdots d_0)=\sum_{j=0}^{r} d_j\, U(j)$.
%Observe that, if $\val_U (d_{r}\cdots d_0) = n$, then the word $d_r \cdots d_0$ is a $U$-representation of $n$ (but not necessarily its normal $U$-representation).

 \begin{example}
 \label{ex:integer-phi-et-phi-sqr-num-sys}
     We resume~\cref{ex:integer-phi-et-phi-sqr}.
     For the golden ratio, we obtain $m=2$, so $U_\varphi(0)=1$, $U_\varphi(1)=1 \cdot U_\varphi(0) + 1=2$ and $U_\varphi(n) = 1 \cdot U_\varphi(n-1) + 1 \cdot U_\varphi(n-2)$ for all $n\ge 2$.
     We find back the Zeckendorf numeration system of~\cref{ex:Fibonacci-Zeckendorf-ANS}.
     For its square, we have $m=1=k$, so $U_{\varphi^2}(0)=1$, $U_{\varphi^2}(1)=2 \cdot U_{\varphi^2}(0) + 1=3$ and $U_{\varphi^2}(n) = 3 \cdot U_{\varphi^2}(n-1) - 1 \cdot U_{\varphi^2}(n-2)$ for all $n\ge 2$.
     Its first few elements are thus $1,3,8,21,55$, the even-indexed Fibonacci numbers. 
 \end{example}

The numeration system $U_\beta$ has two useful properties.
First, it possesses the \emph{Bertrand property}, i.e., for all $w \in A^+_{U_\beta}$, $w \in L_\beta \Leftrightarrow w0 \in L_\beta$; see~\cite{Ber89} for the original version and~\cite{CCS2021} for the recently corrected one.
In that case, any word $w$ in the set $0^*L_\beta$ of all $U_\beta$-representations with leading zeroes is the label of a path in the automaton $\mathcal{A}_\beta$ from~\cref{pro:auto}~\cite[Theorem~2]{CCS2021}.
Second, it satisfies the \emph{dominant root condition}, i.e., there is some complex number $c$ such that
\begin{align}
\lim\limits_{n \rightarrow +\infty} \frac{U_{\beta}(n)}{\beta^n} = c \label{eq:conv-beta-U};
\end{align}
see, for instance,~\cite[Theorem~2]{CCS2021}.
In the following, we will consider \emph{Parry--Bertrand automatic} sequences, which are defined as $S$-automatic for the abstract numeration system built on the language $L_\beta$ of $U_\beta$ for some Parry number $\beta$.
For the sake of conciseness, we call them \emph{$U_\beta$-automatic}.
Notice that when $\beta$ is an integer $b$, the corresponding sequences are usually called \emph{$b$-automatic}. 

\begin{example}
\label{ex:Fibonacci-sum-of-digits}
    In the Zeckendorf numeration system $F$ from~\cref{ex:Fibonacci-Zeckendorf-ANS}, if we write $\rep_F(n) = n_\ell \cdots n_0$ for each non-negative integer $n$, then the \emph{sum-of-digit function $\infw{s}_F \colon \N \to \N$} is defined by $\infw{s}_F (n)= \sum_{i=0}^\ell n_i$.  
    When modded out by $2$, it is $F$-automatic (or Fibonacci-automatic) as the DFAO in~\cref{fig:Fibonacci-DFAO} produces it.
    Note that it is~\cite[A095076]{Sloane}.
\begin{figure}
\begin{center}
\begin{tikzpicture}
\tikzstyle{every node}=[shape=circle,fill=none,draw=black,minimum size=20pt,inner sep=2pt]
\node(a) at (0,0) {$0$};
\node(b) at (2,0) {$1$};

\tikzstyle{every node}=[shape=circle,fill=none,draw=none,minimum size=10pt,inner sep=2pt]
\node(a0) at (-1,0) {};
\node(a1) at (0,1) {$0$};

\tikzstyle{every path}=[color =black, line width = 0.5 pt]
\tikzstyle{every node}=[shape=circle,minimum size=5pt,inner sep=2pt]
\draw [->] (a0) to [] node [] {}  (a);

\draw [->] (a) to [loop above] node [] {$0$}  (a);
\draw [->] (b) to [loop above] node [] {$0$}  (b);

\draw [->] (a) to [bend left] node [above] {$1$}  (b);
\draw [->] (b) to [bend left] node [below] {$1$}  (a);
;
\end{tikzpicture}
\caption{A DFAO producing the sum-of-digit function $\infw{s}_F$ modulo $2$ in the Zeckendorf numeration system $F$.}
\label{fig:Fibonacci-DFAO}
\end{center}
\end{figure}
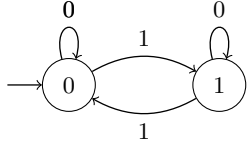
\end{example}

% \begin{example}
%     From~\cref{ex:integer-phi-et-phi-sqr-num-sys}, recall that $U_{\varphi^2}=1,3,8,21,55,\ldots$. The first few words in $L_{\varphi^2}$ are $\eps$, $1$, $2$, $10$, $11$, $12$, $20$, $21$, $100$, $101$, $102$, $110$, $111$, $112$, $120$, $121$, $200$, $201$, $202$, $210$, $211$.
    %For a language $L$, we let $L[n]$ denote the set of length-$n$ words of $L$.
    %The first few words in $L_{\varphi^2}$ and $0^*L_{\varphi^2}$ are given in~\cref{tab:phi-sqrd-words}.
%
%\begin{table}
%\renewcommand{\arraystretch}{1.5}
%\begin{center}
 %   \begin{tabular}{c|c|l|l}
  %       ~$n$~ & ~$U_{\varphi^2}(n)$~ & ~$L_{\varphi^2}[n]$ & ~$(0^{*}L_{\varphi^2})[n]$ \\ \hline 
   %       0 & 1 & ~$\varepsilon$ & ~$\varepsilon$ \\ 
    %      1 & 3 & ~$1,2$ & ~$0,1,2$ \\
     %     2 & 8 & ~$10,11,12,20,21$ & ~$00,01,02$\\ 
     %       &   &                   & ~$10,11,12,20,21$ \\
     %     3 & 21 & ~$100,101,102,110,111$~ &  ~$000,001,002$ \\ 
     %     & &~$112,120,121,200$~ & ~$010,011,012,020,021$ \\
     %     & &~$201,210,211$~ & ~$100,101,\ldots,210,211$ \\
    %\end{tabular}
%\end{center}
%    \caption{The first few words in $L_{\varphi^2}$ and $0^*L_{\varphi^2}$.}
 %   \label{tab:phi-sqrd-words}
%\end{table}
% \end{example}

%%%%%%%%
%%%%%%%%
%%%%%%%%
%%%%%%%%
%%%%%%%%
%%%%%%%%
%%%%%%%%
%%%%%%%%
%%%%%%%%
%%%%%%%%
%%%%%%%%
%%%%%%%%

\section{Main results} \label{sec:our result}

As stated in the introduction, the main purpose of this paper resides in~\cref{thm:main}, which extends~\cref{thm:AS_MW} to a larger family of morphic sequences by considering all even-order correlations. 
Such a result will also imply that this family of morphic sequences cannot be qualified as pseudorandom, such as classical automatic sequences, from the point of view of correlations. 

%\begin{theorem}
%\label{thm:main}
%Let $d\geq 1$ be an integer, $\beta \in \mathbb{R}_{>1}$ be a Parry number and consider the canonical Parry--Bertrand numeration $U_\beta$.
%For all binary $U_\beta$-automatic sequences $\infw{s}$, we have $C_{2d}(\infw{s},N) \gg N$ for all sufficiently large $N$, where the implied constant only depends on $\beta$ and $d$. 
%\end{theorem}

\subsection{Linearly recurrent sequences}
\label{sec:linearly rec seq}

Before developing the arguments for~\cref{thm:main}, we first deal with so-called linearly recurrent sequences.
A sequence $\infw{s}$ is \emph{recurrent} if any factor of $\infw{s}$ occurs infinitely often in $\infw{s}$.
It is moreover \emph{uniformly recurrent} if the distance between two consecutive occurrences of any factor $x$ in $\infw{s}$ is bounded by a constant that depends only on $x$.
For instance, the Thue-Morse sequence is uniformly recurrent.
%and Fibonacci sequences are uniformly recurrent.
%When the gaps are at most linear, we have the following notion: 
We say that $\infw{s}$ is \emph{linearly recurrent} if the gap between two consecutive occurrences of any length-$n$ factor $x$ in $\infw{s}$ is bounded by $C\cdot n$, where $C$ is a constant.
See~\cite{ACSZ2019} for more examples.

\begin{theorem}
\label{thm:linearly recurrent sequence}
    Let $\infw{s}$ be a linearly recurrent sequence. For all $k\geq 1$, we have $C_{2k}(\infw{s},N)\gg N$.  
\end{theorem}

\begin{proof}
    Let $C>0$ be such that two consecutive occurrences of any factor $x$ of $\infw{s}$ being at positions $i,j$ satisfies $\lvert j-i \lvert <C|x|$. Let $N\geq 1$ be sufficiently large and consider the prefix $x=s(0)s(1)\ldots s(\lfloor \frac{N}{C+1} \rfloor -1)$ of $\infw{s}$. By assumption, there exists $d \leq C\lfloor \frac{N}{C+1} \rfloor$ such that $s(i) = s(d+i)$ for all $i\in\{0,\ldots,\lfloor \frac{N}{C+1} \rfloor -1\}$ and $d+\lfloor \frac{N}{C+1} \rfloor -1  \leq N$. Therefore, for $D=(0,d)$ and $M=\lfloor \frac{N}{C+1}\rfloor$, we have $\sum_{n=0}^{M-1}(-1)^{s(n)+s(n+d)}=\sum_{n=0}^{M-1} (-1)^{2s(n)}=M$. This proves that $C_{2}(\infw{s},N)\gg N$. 
    
    The case of even-order of correlations can be treated similarly by considering $2k$ consecutive occurrences of the prefix $x=s(0)s(1)\ldots s(\frac{1}{2k}\lfloor \frac{N}{C+1} \rfloor -1)$, $d\leq \frac{C}{2k}\lfloor \frac{N}{C+1} \rfloor$ such that $s(i)=s(jd+i)$ for all $j\geq 0$ and $i\in\{0,\ldots,\lfloor \frac{N}{C+1} \rfloor -1\}$, and $D=(0,d,\ldots,(2k-1)d)$. 
    \qed
 \end{proof}

%\subsection{Periodic sequences}
%\label{sec:periodic}

%Before developing the arguments for~\cref{thm:main}, we first deal with the case of periodic sequences, defined as the infinite repetition of a pattern $uuu\cdots$ and which are automatic in any numeration system.
%Even though this example seems relatively simple, it highlights a clear difference between even and odd-order correlations, as shown in~\cref{rk:even order vs odd order}.

%\begin{proposition}
%Let $\infw{s}$ be a binary periodic sequence. For all $k\geq 2$ and all sufficiently large $N$, we have $C_{2k}(\infw{s},N) \gg N$.
%\end{proposition}

%\begin{proof}
%Let $T$ be a period of the sequence $\infw{s}$.
%%We divide the proof into two cases depending on the parity of $k$. First, assume that $k$ is even.
%Let $N\geq 1$ be sufficiently large, $M=N-(2k-1)T$, and $D=(d_1,\ldots,d_{2k})=(0,T,\ldots,(2k-1)T)$.
%Note that $M + d_{2k} \le N$.
%Since $T$ is a period of $\infw{s}$, we have 
%\[
%\sum_{n=0}^{M-1}(-1)^{s_{n+d_1}+s_{n+d_2}\cdots %+s_{n+d_{2k}}}=\sum_{n=0}^{M-1} (-1)^{2k s_n}=M. 
%\]
%Thus, $C_{2k}(\infw{s},N) \geq |V(\infw{s},N,M,D)|=M\gg N$, as %desired. \qed 
%\end{proof}

\begin{remark}
\label{rk:even order vs odd order}
Surprisingly, the case of odd-order correlations cannot be treated similarly.
To highlight this, we consider the case of \emph{periodic} sequences, defined as the infinite repetition $u^\omega=uuu\cdots$ of a (finite) pattern $u$ and which are automatic in any numeration system.
More specifically, consider the simple periodic sequence $\infw{s}=(01)^\omega$.
For all $k\geq 1$ and all $N\ge 1$, we have $C_{2k+1}(\infw{s},N)=1$ since, for $D=(d_1,\ldots,d_k)$, $|V(\infw{s},M,D)|$ equals $1$ if $M$ is odd, $0$ otherwise.
% \begin{align*}
%     \left| \sum_{n=0}^{M-1}(-1)^{s_{n+d_1}+s_{n+d_2}\cdots +s_{n+d_{2k+1}}} \right|=\begin{cases}
%         1, & \text{if $M$ is odd}; \\ 0, &\text{otherwise}.
%     \end{cases} 
% \end{align*}
Actually, this simple observation explains why the statement of our main result does not hold for odd-order correlations. 
\end{remark}

Note that there exist automatic sequences that are not uniformly recurrent (and therefore not linearly recurrent). 
For instance, consider the $3$-automatic sequence $\infw{ca}=ababbbababbbbbbbbbabab\cdots$ defined as the fixed point of the uniform morphism $a\mapsto aba,b\mapsto bbb$ (sometimes referred to as the \emph{Cantor sequence}; see, for instance,~\cite[Section~2.4.6]{Walnut}) and for which consecutive occurrences of the factor $aba$ are separated by larger and larger gaps.
Therefore,~\cref{thm:main} covers a larger class of sequences than~\cref{thm:linearly recurrent sequence}.

\subsection{Proof of \Cref{thm:main}}
\label{sec:proof of main}

Now we attack the proof of~\cref{thm:main} for which we need some additional preparation.
First, we count the number of words of each length in the numeration language with leading zeroes, since this language will appear later on.

\begin{lemma}
\label{lem:nbr words fixed len}
Let $\beta \in \mathbb{R}_{>1}$ be a Parry number and consider the canonical Parry--Bertrand numeration, $U_\beta$ whose numeration language is $L_\beta$.
For all $n\ge 0$, the number of length-$n$ words in $0^*L_\beta$ is $U_\beta(n)$.
\end{lemma}
\begin{proof}
    First, for all $m\geq 1$, the length-$m$ words in $L_\beta$ correspond to the integers in $[U_\beta(m-1),U_\beta(m))$, so there are $U_\beta(m)-U_{\beta}(m-1)$ of them.
    The number of length-$n$ words in $0^*L_\beta$ is thus equal to
    \[
    \left(\sum_{m=1}^{n-1} U_\beta(m)-U_{\beta}(m-1) \right) +1 
    =U_\beta(n).
    \]
    Therefore, the lemma is proved. \qed
\end{proof}

Since automata producing Parry--Bertrand automatic sequences recognize the language $L_{\beta}$, we build up one from that of~\cref{pro:auto}. 

\begin{lemma} \label{lem:automate Lb}
    Let $\beta \in \mathbb{R}_{>1}$ be a Parry number. Then there exists a deterministic finite automaton recognizing the language $L_{\beta}$. 
\end{lemma}

\begin{proof}
    Let $\mathcal{A}_{\beta}=\left(Q,a_0,A_{\beta},\delta,Q\right)$ be the DFA recognizing the language $0^{*}L_{\beta}$ from~\cref{pro:auto} with state-set $Q=\{a_0,\ldots,a_{m+k-1}\}$.
    The automaton $\mathcal{A}_{\beta}'=(Q',a_0',A_{\beta},\delta',Q')$ defined as follows recognizes $L_{\beta}$.
    We set $Q'=\{a_0'\}\cup Q$ where $a_0'$ is a new state.
    We replace the initial state $a_0$ of $\mathcal{A}_{\beta}$ by $a_0'$.
    The new transition function $\delta'$ is defined as $\delta'(a_0',c)=\delta(a_0,c)$ for all non-zero letter $c$, and $\delta'(a_i,c)=\delta(a_i,c)$ for all $i\in\{0,\ldots,m+k-1\}$ and all letter $c$. 
    \qed
\end{proof}

It is well known due to Cobham that $b$-automatic sequences are characterized by $b$-uniform morphisms; see~\cite{Cobham72} or~\cite[Theorem~6.3.2]{AS03}.
We have the following generalization for $S$-automatic sequences.

\begin{theorem}[\cite{RM2002}]
\label{thm:morphic-iff-automatic}
A word is morphic if and only if it is $S$-automatic for some abstract numeration system $S$.
\end{theorem}

The proof of the above theorem is constructive: given the morphisms producing the word $\infw{s}$, one can build an abstract numeration system $S$ and a DFAO generating~$\infw{s}$, and vice versa.
As our proof of~\cref{thm:main} relies on this construction, we develop it now.

Fix a Parry number $\beta \in \mathbb{R}_{>1}$ and consider the canonical Parry--Bertrand numeration $U_\beta$.
For the sake of conciseness, let $L$ be the numeration language $L_\beta$ and let $A$ denote the (minimal) alphabet over which $L$ is defined.
Order the letters of $A$ as $0<1<\cdots < m$.
Consider also a binary $U_\beta$-automatic sequence $\infw{s}$.
Let $\mathcal{A}= (Q,a_0,A,\delta_\mathcal{A},F)$ be the DFA accepting $L$ as in~\cref{lem:automate Lb} and let $\mathcal{B} = (R,r_0,A,\delta_\mathcal{B},\tau\colon R\to \{0,1\})$ be a DFAO generating the sequence $\infw{s}$.
% \textbf{here} Note that $F=Q$ as all states are final in~\cref{pro:auto}. \pierre{problem here ?}
%and by~\cref{lem:loop-on-initial-state} we may assume that there is a loop on $r_0$ with label $0$.

We explicitly give a pair of morphisms producing $\infw{s}$.
Define the Cartesian product automaton $\mathcal{P}=\mathcal{A}\times\mathcal{B}$, which imitates the behavior of $\mathcal{A}$ on the first component and $\mathcal{B}$ on the second.
Its set of states is $Q\times R$, the initial state is $(a_0,r_0)$, and the alphabet is $A$.
Its transition function $\Delta \colon (Q\times R)\times A^* \to Q\times R$ is defined by $\Delta((q,r),w)=(\delta_\mathcal{A}(q,w),\delta_\mathcal{B}(r,w))$ for all $q\in Q$, $r\in R$ and $w\in A^*$.
In particular, $\Delta((q,r),w)\in F\times R$ if and only if $w$ belongs to $L$.
Furthermore, if $\Delta((a_0,r_0),\rep_{U_{\beta}}(n))=(q,r)$, then $\infw{s}(n) = \tau(r)$.
Now let $\alpha$ be a symbol not belonging to $Q \times R$.
Define the morphism $\varphi_\mathcal{P} \colon (Q\times R \cup \{\alpha\})^* \to (Q\times R \cup \{\alpha\})^*$ by $\varphi_\mathcal{P}(\alpha) = \alpha (a_0,r_0)$ and, for all $q\in Q$ and $r\in R$, 
    \[
    \varphi_\mathcal{P}((q,r))=\Delta((q,r),0)\cdots \Delta((q,r),m) = \prod_{i=0}^m \Delta((q,r),i), 
    \]
where, if $\Delta((q,r),i)$ is not defined for some $i$, then it is replaced by $\eps$ (note that the product has to be understood as the concatenation).
Observe that $\varphi_\mathcal{P}$ is not necessarily uniform.
%Consider that all the states of $\mathcal{P}$ to be finite and let $L_\mathcal{P}$ denote the language accepted by $\mathcal{P}$.
%Order the words in $L_\mathcal{P}$ genealogically as $w_0 < w_1 < \cdots$.
Order the words in $L$ genealogically as $w_0 < w_1 < \cdots$.
Define $\infw{u}_\mathcal{P}=\varphi_\mathcal{P}^\omega(\alpha)$.
Then, by~\cite[Lemma~2.25]{Rigo14-2}, the shifted sequence of $\infw{u}_\mathcal{P}$ is the sequence of states reached in $\mathcal{P}$ by the words in $L$ ordered genealogically, i.e., for all $n\in\N$, the $(n+1)$st letter of $\infw{u}_\mathcal{P}$ is equal to $\Delta((a_0,r_0),w_n)$.
Now define the morphism $\nu_\mathcal{P} \colon (Q\times R \cup \{\alpha\})^* \to \{0,1\}^*$ by $\nu_\mathcal{P}(\alpha)=\eps$ and, for all $q\in Q$ and $r\in R$, $\nu_\mathcal{P}((q,r))=\tau(r)$. Note that $\nu_\mathcal{P}$ is non-erasing on $Q\times R$.
By construction, we have $\infw{s}=\nu_\mathcal{P}(\infw{u}_\mathcal{P})=\nu_\mathcal{P}(\varphi_\mathcal{P}^\omega(\alpha))$.
To help the reader with the construction, we illustrate it in~\cref{ex:example non Parry but Bertrand} below.

We now give the proof of~\cref{thm:main}. In particular, one side feature of the proof is a characterization of the distance between the repeating part in the sequence under study.

\begin{proof}[Proof of~\cref{thm:main}]
Let $L$ denote the numeration language $L_\beta$ and let $A$ be the (minimal) alphabet over which $L$ is defined.
%Order the letters of $A$ as $0<1<\cdots < m$. 
Let $\mathcal{A}= (Q,a_0,A,\delta_\mathcal{A},F)$ be the DFA accepting $L$ as in~\cref{lem:automate Lb}, $\mathcal{B} = (R,r_0,A,\delta_\mathcal{B},\tau\colon R\to \{0,1\})$ be a DFAO generating the sequence $\infw{s}$, $\mathcal{P}=\mathcal{A}\times\mathcal{B}$ be the Cartesian product automaton, with transition function $\Delta \colon (Q\times R)\times A^* \to Q\times R$, and $\alpha$ be a symbol not belonging to $Q \times R$.
Consider the morphisms $\varphi_\mathcal{P} \colon (Q\times R \cup \{\alpha\})^* \to (Q\times R \cup \{\alpha\})^*$ and $\nu_\mathcal{P} \colon (Q\times R \cup \{\alpha\})^* \to \{0,1\}^*$ as defined above and write $\infw{u}_\mathcal{P}=\varphi_\mathcal{P}^\omega(\alpha)$.
    
Let $N,M\geq 1$ where $M$ is chosen later accordingly to $N$.
Consider two distinct words $u,v$ such that $\Delta((a_0,r_0),u)=\Delta((a_0,r_0),v)$. Then, for all length-$M$ words $w\in 0^*L$, the states $\Delta((a_0,r_0),uw)$ and $\Delta((a_0,r_0),vw)$ are equal.
%are both equal to $\Delta((a_0,r_0),w)$ by assumption on the initial state $(a_0,r_0)$
The word $\varphi_\mathcal{P}^{M}(\Delta((a_0,r_0),u))=\varphi_\mathcal{P}^{M}(\Delta((a_0,r_0),v))$ is made of the sequence of such states.
Indeed, for $x\in \{u,v\}$, we have
\[
\varphi_\mathcal{P}^{M}(\Delta((a_0,r_0),x))
=\prod_{\substack{w\in 0^*L\\ |w|=M}}\Delta((a_0,r_0),xw).
\]
%\begin{align*}
%\varphi_\mathcal{P}^{M}(\Delta((a_0,r_0),u))&=\prod_{\substack{w\in 0^*L\\ |w|=M}}\Delta((a_0,r_0),uw), \\&=\prod_{\substack{w\in 0^*L\\ |w|=M}}\Delta((a_0,r_0),vw)= \varphi_\mathcal{P}^{M}(\Delta((a_0,r_0),v)).
%\end{align*}
Applying the morphism $\nu_{\mathcal{P}}$ then gives back the original sequence $\infw{s}$.
As $\nu_{\mathcal{P}}$ is non-erasing on $Q\times R$, $|\varphi_\mathcal{P}^{M}(\Delta((a_0,r_0),x))|=U_\beta(M)$ by~\cref{lem:nbr words fixed len}.

\noindent \textbf{Order $2$.} Firstly, we study the case of the order-$2$ correlation, for the sake of simplicity and since the case of even-order correlations will follow the same lines. 
Consider two distinct words $u,v\in L$ such that $\Delta((a_0,r_0),u)=\Delta((a_0,r_0),v)$. 
By the pigeonhole principle, we may choose $u,v$ among the first $|Q|\cdot|R|$ words of $L$.
Without loss of generality, we may assume that $u$ is lexicographically smaller than $v$ and that $u$ is non-empty (recall~\cref{lem:automate Lb}).
Then in the sequence $\infw{s}$ we have two identical blocks of length $U_{\beta}(M)$ starting at positions $p_u,p_v$ respectively associated with the word $u0^M$ and $v0^M$ in $L$, i.e., $p_x=\val_{U_\beta}(x0^M)$ for $x\in\{u,v\}$.
The situation is depicted in~\cref{fig:proof of main thm}.
Therefore, we have 
\begin{align}\label{eq: equality order 2}
 \infw{s}(p_u + n)=\infw{s}(p_v + n)   
\end{align}
for all  $0\leq n <U_{\beta}(M)$.
We now make several observations.
First, notice that the largest index in~\cref{eq: equality order 2} is smaller than $p_v+U_{\beta}(M)$ by assumption on $u,v$.
    Then, note that, for $x\in\{u,v\}$, we have $p_x=\val_{U_\beta}(x0^M)$ and $\val_{U_\beta}(x) \le |Q|\cdot|R|$ by choice of $x$.

Let $N\geq 1$ be sufficiently large. Choose $M\geq 1$ such that $p_v+U_{\beta}(M)\leq N< p_v+U_{\beta}(M+1)$ and set $D=(p_u,p_v)$. Therefore, using~\cref{eq: equality order 2}, we have
\begin{equation}
    C_2(\infw{s},N)\geq \left| V(\infw{s},N,U_{\beta}(M),D) \right|
    =\sum_{n=0}^{U_{\beta}(M)-1} (-1)^{\infw{s}(p_u + n)+\infw{s}(p_v + n)}
    =U_{\beta}(M).  \label{eq:final_order_2}
\end{equation}
On the other hand, we have $U_{\beta}(M)=\Theta(\beta^M)$ by~\cref{eq:conv-beta-U}, which yields
%Let $\beta_0$ be such that $\beta>\beta_0>1$. Then
%\begin{align*}
%\lambda_{M+1}
%&=\sum_{i=0}^{M+2} U_{\beta}(i)+U_{\beta}(M+1)\ll \sum_{i=0}^{M+2} \beta^{i}+\beta^{M+1} \\ 
%& \ll \frac{\beta^{M+3}-1}{\beta-1}+\beta^{M+1} \\ 
%& \ll \frac{1}{\beta_0-1}(\beta^{M+3}-1)+\beta^{M+1} \ll \frac{\beta^{3}\beta_0}{\beta_0-1} \; \beta^M. 
%\end{align*} Then~\cref{eq:conv-beta-U,eq:final_order_2} leads to \begin{align*}
%    C_2(\infw{s},N) \geq U_{\beta}(M) \gg \beta^{M} \gg \frac{\beta_0-1}{\beta_3\beta_0}\lambda_{M+1}\gg N. 
%\end{align*} This concludes the case of the order-$2$ correlation.
\begin{equation}\label{eq:bound on pv}
    p_v+U_{\beta}(M+1)
\le \lceil\beta\rceil^{M+1} \cdot |Q|\cdot|R| + U_{\beta}(M+1)
\ll c \beta^{M+1}
\end{equation}
for some non-zero constant $c$ depending on $\beta$ and the sequence $\infw{s}$.
Then~\cref{eq:final_order_2,eq:bound on pv} lead to \begin{align*}
    C_2(\infw{s},N) \geq U_{\beta}(M) \gg \beta^{M} \gg \frac{p_v+U_{\beta}(M+1)}{c\beta}\gg N 
\end{align*}
by choice of $M$.
This concludes the case of the order-$2$ correlation. 

\begin{figure}
    \centering
    \begin{tikzpicture}[cross/.style={path picture={ 
    \draw[black]
    (path picture bounding box.south east) -- (path picture bounding box.north west) (path picture bounding box.south west) -- (path picture bounding box.north east);
    }}]  
        \node (a) at (-0.5,0) {$\infw{u}_\mathcal{P}$};
        \draw [-] (0,0.1) -- (0,-0.1);
        \draw [-] (0,0) -- (7,0);
        \draw [dotted] (7,0) to (7.5,0);
        \node [cross] (1) at (1,0) {};
        \node (a) at (1,0.3) {$u$};
        \node [cross] (2) at (2,0) {};
        \node (a) at (2,0.3) {$v$};

        \draw [->] (-0.75,0) to [out=180,in=130] (-0.75,-1);
        \node (e) at (-1.25,-0.5) {$\varphi_\mathcal{P}^M$};

        \draw [dotted] (1) -- (1.25,-0.85);
        \draw [dotted] (1) -- (2.75,-0.85);
        \draw [dotted] (2) -- (4,-0.85);
        \draw [dotted] (2) -- (5.5,-0.85);

        \node (b) at (-0.5,-1) {$\infw{u}_\mathcal{P}$};
        \draw [-] (0,-0.9) -- (0,-1.1);
        \draw [-] (0,-1) -- (7,-1);
        \draw [dotted] (7,-1) to (7.5,-1);
        \draw [draw=black] (1.25,-1.15) rectangle (2.75,-0.85);
        \draw [draw=black] (4,-1.15) rectangle (5.5,-0.85);

        \node (c) at (-0.5,-2) {$\infw{s}$};
        \draw [-] (0.25,-1.9) -- (0.25,-2.1);
        \draw [-] (0.25,-2) -- (7,-2);
        \draw [dotted] (7,-2) to (7.5,-2);
        \draw [draw=black] (1.25,-2.15) rectangle (2.75,-1.85);
        \draw [draw=black] (4,-2.15) rectangle (5.5,-1.85);
        
        \draw [->] (-0.75,-1) to [out=220,in=180] (-0.75,-2);
        \node (d) at (-1.25,-1.5) {$\nu_\mathcal{P}$};

        \draw [dotted] (1.25,-1.15) -- (1.25,-1.85);
        \draw [dotted] (2.75,-1.15) -- (2.75,-1.85);
        \draw [dotted] (4,-1.15) -- (4,-1.85);
        \draw [dotted] (5.5,-1.15) -- (5.5,-1.85);
        \node (g1) at (1.25,-2.15) [below] {$p_u$};
        \node (g1) at (4,-2.15) [below] {$p_v$};
        
        \draw [<-] (1.25,-1.5) to (1.5,-1.5);
        \draw [->] (2.5,-1.5) to (2.75,-1.5);
        \node (f) at (2,-1.5) {$U_\beta(M)$};

        \draw [<-] (4,-1.5) to (4.25,-1.5);
        \draw [->] (5.25,-1.5) to (5.5,-1.5);
        \node (f) at (4.75,-1.5) {$U_\beta(M)$};
    \end{tikzpicture}
    \caption{A representation of the situation depicted in the proof of~\cref{thm:main}.}
    \label{fig:proof of main thm}
    \end{figure}
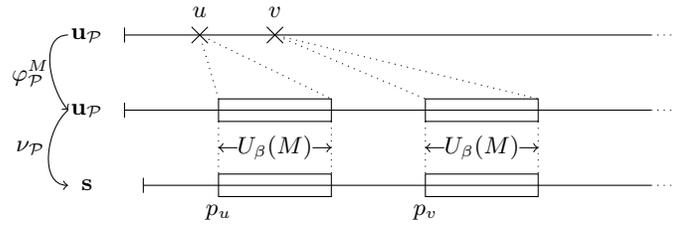

\noindent \textbf{Even orders.} The general case of even-order correlations can be treated similarly. 
Indeed, let $k\geq 1$.
Consider $2k$ pairwise distinct words $u_1,\ldots,u_{2k}\in L$ such that $\Delta((a_0,r_0),u_i)=\Delta((a_0,r_0),u_j)$ for all $i,j \in\{1,\ldots,2k\}$. 
By the pigeonhole principle, we may choose them among the first $(2k-1) \cdot |Q|\cdot|R|$ words of $L$.
Without loss of generality, we may assume that they are increasingly nested lexicographically speaking (and non-empty).
Then in the sequence $\infw{s}$ we have $2k$ identical blocks of length $U_{\beta}(M)$ starting at positions $p_{u_i}=\val_{U_\beta}(u_i0^M)$ for $i\in\{1,\ldots,2k\}$.
Therefore, we have $\infw{s}(p_{u_i} + n)=\infw{s}(p_{u_j} + n)$ for all $1\leq i\leq j \leq 2k$ and $0\leq n <U_{\beta}(M)$.
We now finish up the proof by following the same argument as in the previous case.
\qed
\end{proof}

\section{Discussion and future work}\label{sec:conclusion}

We can actually obtain a lower bound on correlations using the factor complexity, i.e., the map counting the number of different factors of each length within a given sequence.
Consider a sequence $\infw{s}$ that is not ultimately periodic.
Let $p_{\infw{s}} \colon \mathbb{N} \to \mathbb{N}$ denote its factor complexity.
Let $k,N$ be integers with $k\ge 1$.
Then, using Morse and Hedlund's famous result~\cite{MorseHedlund1938}, one can show that, within the prefix of length $k p_{\infw{s}}(N)$ of $\infw{s}$, there exists a length-$N$ factor that occurs at least $k$ times.
Therefore, we obtain the lower bound $C_{2k}(\infw{s},2k p_{\infw{s}}(N)) \geq N$.
In particular, for Parry--Bertrand automatic sequences,~\cite[Theorem~3.4]{MPR2019} allows to obtain the behavior emphasized in~\cref{thm:main}.  

However, the aforementioned lower bound does not reflect the expected behavior of $C_k$ when the studied sequence is random, or in some other simple cases.
For instance, for the binary Champernowne sequence $\infw{ch}=0 1 10 11 100 101 110 111 \cdots$, we have $C_2(\infw{ch},N)>\frac{1}{48}N$ (see \cite[Theorem~1]{MS-II}), whereas $p_{\infw{ch}}(N)=2^N$ for all sufficiently large $N$. 

Furthermore, ~\cref{thm:main} extends to numeration systems that are not Parry-Bertrand.
In fact, scrutinizing its proof, we need a good control on the proportion of letters erased under the morphism $\nu_{\mathcal{P}}$.
Indeed, the proportion of non-erased letters is precisely the length of the block that is represented in~\cref{fig:proof of main thm}. 
%In the case under study, the morphism only erases one letter.
So, in turn, a good knowledge on the growth rate of the numeration numeration language, i.e., the number of words of each length within the language, is required.
In the next example, using the method presented in the proof of~\cref{thm:main}, we exhibit an automatic sequence having a quadratic factor complexity but a linear order-$2$ correlation.

\begin{example}\label{ex:example non Parry but Bertrand}
Consider the linear positional numeration system $U=(U(n))_{n\ge 0}$ defined by $U(0)=1$ and $U(n+1)=3U(n)+1$ for all $n\ge 0$. It has several useful properties: its numeration language is given by the DFA in~\cref{fig:product ex} and $U$ is Bertrand but not Parry; see~\cite[Example~2 and Lemma~2.5]{MPR2019}, as well as~\cite[page~131]{Hollander1998}.
We also have $U(n) =\Theta(\beta^n)$ where $\beta = \frac{1}{2} (3+\sqrt{13}) \approx 3.30278$.
Consider the fixed point $\infw{x}$ of the morphism $a\mapsto aaab$, $b\mapsto b$ and the coding $\tau \colon a\mapsto 0, b\mapsto 1$. The sequence $\tau(\infw{x})$ is produced by the DFAO in~\cref{fig:product ex}. It is binary and $U$-automatic.
In addition, its factor complexity is known to be quadratic (see~\cite[Theorem~3.3]{MPR2019}).

\begin{figure}
\centering
\begin{tikzpicture}
\tikzstyle{every node}=[shape=circle,fill=none,draw=black,minimum size=20pt,inner sep=2pt]
\node(a) at (1,-2) {$a/0$};
\node(b) at (1,-4) {$b/1$};

\tikzstyle{every node}=[shape=rectangle,fill=none,draw=none]
\node(a0) at (1,-1) {};

\tikzstyle{every path}=[color =black, line width = 0.5 pt]
\draw [->] (a0) to [] node [] {}  (a);

\draw [->] (a) to [loop left] node [] {$0,1,2$}  (a);
\draw [->] (a) to [] node [right] {$3$}  (b);

\draw [->] (b) to [loop left] node [] {$0$}  (b);
;

\tikzstyle{every node}=[shape=circle,fill=none,draw=black,minimum size=20pt,inner sep=2pt]
\node[accepting](a0') at (3,0) {$a_0'$};
\node[accepting](a0) at (5.5,0) {$a_0$};
\node[accepting](a1) at (8,0) {$a_1$};

\tikzstyle{every node}=[shape=rectangle,fill=none,draw=none]
\node(a00) at (2,0) {};

\tikzstyle{every path}=[color =black, line width = 0.5 pt]
\draw [->] (a00) to [] node [] {}  (a0');

\draw [->] (a0') to [] node [above] {$1,2$}  (a0);
\draw [->] (a0') to [bend right] node [below] {$3$}  (a1);

\draw [->] (a0) to [loop above] node [] {$0,1,2$}  (a0);
\draw [->] (a0) to [] node [above] {$3$}  (a1);

\draw [->] (a1) to [loop above] node [above] {$0$}  (a1);
;

\node (f0) at (2,-2) [] {};
\node (A) at (3,-2) [draw,circle] {$A/0$};
\node (B) at (5.5,-2) [draw,circle] {$B/0$};
\node (C) at (8,-4) [draw,circle] {$C/1$};

\draw [->] (f0) to [] node [] {}  (A);

\draw [->] (A) to [] node [above] {$1,2$}  (B);
\draw [->] (A) to [] node [below] {$3$}  (C);

\draw [->] (B) to [loop right] node [] {$0,1,2$}  (B);
\draw [->] (B) to [] node [above] {$3$}  (C);

\draw [->] (C) to [loop above] node [] {$0$}  (C);

\end{tikzpicture}
\caption{The product $\mathcal{P}$ between a DFA for $L_{U}$ (top) and a DFAO for the morphic sequence $\tau(\infw{x})$ (right) where $U$, $\infw{x}$, and $\tau$ are all defined in~\cref{ex:example non Parry but Bertrand}.}
\label{fig:product ex}
\end{figure}
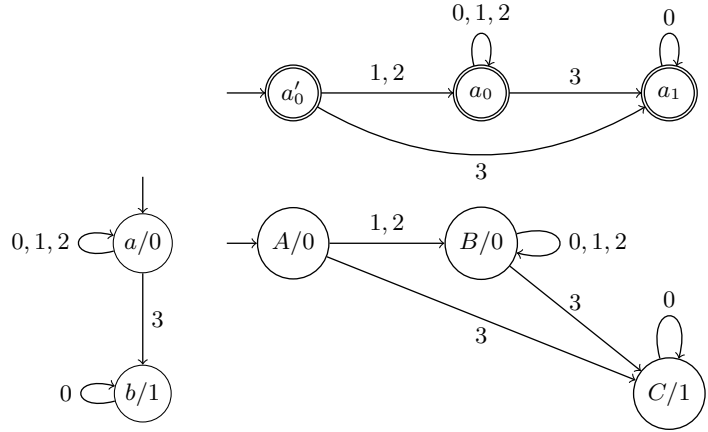

Using the notation of~\cref{sec:proof of main}, the morphism $\varphi_\mathcal{P}$ is defined by $\alpha \mapsto \alpha A$, $A\mapsto BBC$, $B\mapsto BBBC$, $C\mapsto C$.
Note that it is not uniform.
We get $\varphi_\mathcal{P}^\omega(\alpha) = \alpha A BBC BBBC BBBC C BBBC \cdots$, whose shift indeed corresponds to the sequence $A,B,B,C,B,B,B,C,\ldots$ of states reached by the words of $L_{U}=\{\eps, 1, 2, 3, 10,11,12,13,20,21,22,23,30,100,101,102,\ldots\}$ in the product automaton $\mathcal{P}$ of~\cref{fig:product ex}.
The morphism $\nu_\mathcal{P}$ is defined by $\alpha \mapsto \eps$, $A,B\mapsto 0$, and $C\mapsto 1$.
Applying $\nu_\mathcal{P}$ to $\varphi_\mathcal{P}^\omega(\alpha)$ gives the sequence $\tau(\infw{x})$.
To illustrate the method of the proof of~\cref{thm:main}, we can consider $u=10$ and $v=100$ for which $\delta(A,u)=B=\delta(A,v)$.
%For $M\geq 1$, applying $\varphi_{\mathcal{P}}^M$ to $u,v$ leads to two identical blocks in the sequence $\tau(\infw{x})$ of size $U(M)$.
Now we also conclude that $ C_2(\tau(\infw{x}),N) \gg N$.
\end{example}

As a consequence, we raise the following more general and natural question: Is any morphic sequence not pseudorandom as far as the correlation measures is concerned, i.e., is the correlation of any morphic sequence large or small?
In the near future, we hope to revisit the method of the proof of~\cref{thm:main} to extend it to a larger class of morphic sequences.
In particular, as a first step, we will need to find the appropriate conditions to describe this particular class.
On the other hand, as shown in~\cite{MW18}, a lower bound on the state complexity of binary sequences in terms of the order-$2$ correlation measure can be straightforwardly derived from~\cref{thm:AS_MW}.
In our setting, can a refinement of the proof of~\cref{thm:main}, involving only the number of states of some automata, lead to similar bounds?

%%%%%%%%%%%%%%%%%%%%%%%%%%%%%%%
%%%%%%%%%%%%%%%%%%%%%%%%%%%%%%%
%%%%%%      ACKO         %%%%%%
%%%%%%%%%%%%%%%%%%%%%%%%%%%%%%%
%%%%%%%%%%%%%%%%%%%%%%%%%%%%%%%

\subsection*{Acknowledgments}
We thank the reviewers for providing useful remarks that improved our initial work, especially for thinking of the alternative proof mentioned at the beginning of~\cref{sec:conclusion}, which encourages us to look for extensions of our~\cref{thm:main} to a broader family of sequences.
We also thank France Gheeraert and Michel Rigo for insightful discussions.

Pierre Popoli's research is supported by ULiège's Special Funds for Research, IPD-STEMA Program.
Manon Stipulanti is an FNRS Research Associate supported by the Research grant 1.C.104.24F.

\newpage

%%%%%%%%%%%%%%%%%%%%%%%%%%%%%%%
%%%%%%%%%%%%%%%%%%%%%%%%%%%%%%%
%%%%%%      BIBLIO       %%%%%%
%%%%%%%%%%%%%%%%%%%%%%%%%%%%%%%
%%%%%%%%%%%%%%%%%%%%%%%%%%%%%%%


\bibliographystyle{splncs04}
\bibliography{biblio.bib}

\end{document}